\documentclass[12pt,a4paper,openright,reqno]{extarticle}
\usepackage[english]{babel}
\usepackage{newlfont}
\usepackage{amsbsy}
\usepackage[dvips]{graphicx}
\usepackage{color}
\usepackage{bbm}
\usepackage{graphicx, subfigure}
\usepackage{color}
\usepackage[center]{caption2}
\usepackage{amssymb}
\usepackage{amsmath}

\usepackage{latexsym}
\usepackage{amsthm}
\usepackage{amsthm}

\usepackage[dvips]{geometry}
\theoremstyle{plain}
\newtheorem{thm}{Theorem}
\newtheorem{cor}[thm]{Corollary}

\newtheorem{lem}[thm]{Lemma}

\newtheorem{nota}[thm]{Notation}

\theoremstyle{remark}
\newtheorem{rem}[thm]{Remark}

\newcommand{\Z}{\mathbb{Z}}
\newcommand{\R}{\mathbb{R}}\newcommand{\C}{\mathbb{C}}

\newcommand{\bp}{\begin{pmatrix}}
\newcommand{\ep}{\end{pmatrix}}

\usepackage{mathtools}
\def\alphaltiset#1#2{\ensuremath{\left(\kern-.2em\left(\genfrac{}{}{0pt}{}{#1}{#2}\right)\kern-.2em\right)}}
\usepackage[center]{caption2}

\begin{document}

\title{Maximal dimension of affine subspaces of specific matrices }
\author{Elena Rubei}
\date{}
\maketitle

{\footnotesize\em Dipartimento di Matematica e Informatica ``U. Dini'', 
viale Morgagni 67/A,
50134  Firenze, Italia }

{\footnotesize\em
E-mail address: elena.rubei@unifi.it}

\def\thefootnote{}
\footnotetext{ \hspace*{-0.36cm}
{\bf 2010 Mathematical Subject Classification: } 15A30

{\bf Key words:} affine subspaces, nilpotent matrices, normal matrices}

\begin{abstract} 
For every $n \in \mathbb{N}$ and every field $K$,  let $N(n,K)$ be the set of the nilpotent $n \times n$ matrices over  $K$ and let $D(n,K) $ be the set of the  $n \times n$ matrices over  $K$ which are diagonalizable over $K$.
Moreover, let $R(n) $ be the set of the normal $n \times n$ matrices.

In this short note we prove that the maximal dimension of an affine subspace in $N(n,K)$
is $ \frac{n(n-1)}{2}$ and, if the characteristic of the field is zero, 
 an affine not linear  subspace in $N(n,K)$ has dimension  less than or equal to $ \frac{n(n-1)}{2}-1$. 
Moreover we prove that
 the maximal dimension of an affine subspace in $R(n)$ is $n$,
the maximal dimension of a linear subspace in $D(n, \R)$ is  $ \frac{n(n+1)}{2}$, while the 
maximal dimension of an affine not linear subspace in $D(n, \R)$ is  $ \frac{n(n+1)}{2} -1$.
\end{abstract}

\section{Introduction}

There is a wide literature on the maximal dimension of linear or affine subspaces
of matrices with specific characteristics.
In particular we quote the following results.
For every $m,n \in \mathbb{N}$ and every field $K$, let $M(m \times n, K)$ be the vector space of the $m \times n$ matrices over $K$.
 Let $N(n,K)$ be the set of the nilpotent $n \times n $ matrices over  $K$ and 
 let $D(n,K) $ be the set of the  $n \times n$ matrices over  $K$ which are diagonalizable over $K$.
Moreover, let $R(n) $ be the set of the normal $n \times n$ matrices.

\begin{thm} \label{Gerst} [Gerstenhaber] Let $K$ be  a field. The maximal dimension of a linear subspace in  $N(n,K)$ 
 is $\frac{n(n-1)}{2}$. 
\end{thm}

We mention that in \cite{M-O-R}, the authors 
gave a new simple proof of Gerstenhaber's result.
In \cite{Q} and \cite{DP2} the authors generalized  Gerstenhaber's result as follows:

\begin{thm} [Quinlan, De Seguins Pazzis] 
Let $K$ be a field. The maximal dimension of a linear subspace of $n \times n$ matrices  over $K$ with no nonzero eigenvalue  is $\frac{n(n-1)}{2}$. 
\end{thm}

As observed in \cite{DP} the above statement is equivalent to the statement that the maximal dimension of an affine subspace  of invertible  $n \times n$ matrices over $K$ is $\frac{n(n-1)}{2}$. 

We say that an affine subspace $S$ of $M(m \times n, K)$ has constant rank $r$ if every matrix of $S$ has rank $r$ and we say that a linear subspace $S$ 
 of $M(m \times n, K)$ has constant rank $r$ if every nonzero matrix of $S$ has rank $r$.

There are many results on linear subspaces of constant rank. We quote some of them.

\begin{thm}  {\bf (Westwick, \cite{W1})} For $2 \leq r \leq m \leq n$, we have the maximal dimension of a linear subspace in
$M(m \times n, \C)$ with constant rank $r$ is less than or equal to $ m+ n -2 r+1$ and greater than or equal to 
$ n-r+1 $.
\end{thm}

\begin{thm}  {\bf (Ilic-Landsberg, \cite{I-L})} If $r$ is even and greater than or equal to $2$, then
 the maximal dimension of a linear subspace in the space of the complex symmetric 
$n \times n$ matrices with constant rank $r$
is $n-r +1$.
\end{thm}

In case $r$ odd, the following result holds, see \cite{Ga}, \cite{H-P}, \cite{I-L}:

\begin{thm}  If $r$ is odd, then  the maximal dimension of a linear subspace in the space of the complex symmetric 
$n \times n$ matrices with constant rank $r$ is $1$.
\end{thm}

We mention also that,  in \cite{Fl}, Flanders
proved that, if $ r \leq m \leq n$, a linear subspace of $M(m \times n ,\mathbb{C})$ such that every of its elements has rank less than or equal to $r$ has dimension less than or equal to $r n$.

In \cite{Ru} we proved the following theorems:

\begin{thm}\label{mio}  Let $n,r \in \mathbb{N}$ with $r \leq n$.
Then the maximal dimension of an affine subspace in the space of the real symmetric 
$n \times n$ matrices with constant rank $r$ is less than or equal to 
$$
\left\lfloor \frac{r}{2} \right\rfloor
\left(n-  \left\lfloor \frac{r}{2} \right\rfloor\right)
.$$ 
\end{thm}

\begin{thm}\label{mio2}  Let $m,n,r \in \mathbb{N}$ with $r \leq m \leq n$.
Then the maximal dimension of an affine subspace in $M(m \times n, \R)$ with constant rank $r$ is 
$$ rn- \frac{r(r+1)}{2} .$$ 
\end{thm}

In \cite{Ru2} we investigated on the maximal dimension of affine subspaces in the space of the antisymmetric matrices of constant rank.  
Precisely, we proved the following theorem:

\begin{thm} \label{thmantisym}
The maximal dimension of an affine subspace in the space of the antisymmetric real 
$n \times n$ matrices with constant rank $2r$ is 
$$ \left\{ 
\begin{array}{lll} 
(n-r-1) r   &  if  & n \geq 2r+2   
\\
r(r-1)   & if & n=2r 
\\
 r(r+1)   & if  & n=2r+1
\end{array} \right.$$
\end{thm}

In this short note we focus on the maximal dimension of affine subspaces in $N(n,K)$ for any field $K$ and on the maximal dimension of affine subspaces in $R(n)$ and in $D(n, \R)$; precisely we prove the following results.

\begin{thm} \label{nilp} Let $K$ be a field.
The maximal dimension of an affine subspace in $N(n,K)$
is $ \frac{n(n-1)}{2}$. 

Moreover, if the characteristic of the field is zero, 
 an affine not linear  subspace in $N(n,K)$ has dimension  less than or equal to $ \frac{n(n-1)}{2}-1$. 
\end{thm}

\begin{thm} \label{norm} 
The maximal dimension of an affine subspace in $R(n)$
is $ n$. 
\end{thm}

\begin{thm} \label{diag} 
The maximal dimension of a linear subspace in $D(n, \R)$ is  $ \frac{n(n+1)}{2}$, while the 
maximal dimension of an affine not linear subspace in $D(n, \R)$ is  $ \frac{n(n+1)}{2} -1$.
\end{thm}

\section{Proof of the theorems}

\begin{nota} Let $n \in \mathbb{N} -\{0\} $. 
For any field $K$, 
we denote the $n \times n$ identity matrix over $K$ by $I^K_n$.
We omit the superscript and the subscript when it is clear from the context.
 
 We denote by $E^K_{i,j}$ the $n \times n$  matrix, whose $(i,j)$-entry is $1$ and all the other entries are zero.
 We omit the superscript when it is clear from the context.
 
 We denote by $A(n,K)$ the subspace of the antisymmetric matrices of $M(n \times n, K)$. 
 
 

For any square matrix $A$, let $S_2(A)$ be the sum of the $2 \times 2 $ principal minors of $A$.
 
For any complex matrix $A$, let $A^{\ast}$ be the transpose of the conjugate matrix of $A$.

\end{nota}

To prove Theorem \ref{nilp} we follow the guidelines of the proof of Gerstenhaber's result in \cite{M-O-R} and firstly 
we need to generalize some lemmas in  \cite{M-O-R}.

\begin{lem} Let $n \in \mathbb{N} -\{0\} $ and $K$ be a field.
Let $R  \in M(n \times n, K)$ and $U \in N(n,K)$. Then
$$S_2(R)- S_2(R+U)= tr(RU).$$
\end{lem}

\begin{proof}
$\bullet$ First let us suppose that $U$ is in  Jordan form. Then 
$$S_2(R)- S_2(R+U)=$$ $$= \sum_{i,j \in 
\{1,\ldots , n\}, \; i<j} \left(R_{i,i} R_{j,j} - R_{i,j} R_{j,i}\right) - \sum_{i,j \in 
\{1,\ldots , n\}, \; i<j} \left((R+U)_{i,i} (R+U)_{j,j} - (R+U)_{i,j} (R+U)_{j,i}\right) =$$
$$= \sum_{i,j \in 
\{1,\ldots , n\}, \; i<j} \left(R_{i,i} R_{j,j} - R_{i,j} R_{j,i}\right) - \sum_{i,j \in 
\{1,\ldots , n\}, \; i<j} \left(R_{i,i} R_{j,j} - (R+U)_{i,j} R_{j,i}\right) =$$
$$= \sum_{i \in 
\{1,\ldots , n-1\}}  U_{i,i+1} R_{i+1,i}
= \sum_{j\in \{1,\ldots , n\}}  \sum_{i \in \{1, \ldots ,n\}}  R_{j,i} U_{i,j} 
= \sum_{j \in 
\{1,\ldots , n\}} (RU)_{j,j} =  tr(RU).$$

$\bullet$ Now let $U$ be generic. Let $C \in GL(n,K)$ be such that $C^{-1}UC$ is in Jordan form. Then 
$$S_2(R)- S_2(R+U)= S_2(C^{-1} R C)- S_2(C^{-1}(R+U)C)=$$
$$  = S_2(C^{-1} R C)- S_2(C^{-1}RC +C^{-1}UC)= $$
$$= tr (C^{-1}RCC^{-1}UC)  =tr(RU),$$
where in the last but one equality we have used the previous item.
\end{proof}

\begin{lem} \label{generaliz}
Let $n \in \mathbb{N} -\{0\} $ and $K$ be a field.
Let $P,A,B \in M(n \times n, K)$ such that $P, P+A, P+B, P+A+B \in N(n,K)$. 
Then $tr(AB)=0$.
\end{lem}

\begin{proof}
By the previous lemma, we have:
$$ S_2(A)- S_2(A+P)= tr(AP)$$ 
(since $P$ is nilpotent)
and
$$ S_2(A)- S_2(A+B+P)= tr(A(B+P))$$ 
(since $B+P$ is nilpotent). Hence we get:
$$ tr(AB)= tr(A(B+P))  -  tr(AP)=
S_2(A)- S_2(A+B+P) -  S_2(A)+ S_2(A+P)= $$ 
$$=- S_2(A+B+P) + S_2(A+P)= -0+0=0,$$
where the last but one equality holds because $A+P$ and $A+B+P$ are nilpotent.
\end{proof}

The following lemma and corollaries we be useful to prove Theorem \ref{norm} and   probably they are well-known; for the convenience of the reader we include them here.

\begin{lem}  \label{simdiagpre} Let
$\{A_1, \ldots, A_r\}$ be a  subset of $M(n \times n, \C) $ such that $A_i$ and $A_j$ are 
simultaneously diagonalizable for every $i,j \in \{1, \ldots,r\}$; let $E_l^i $ for $l=1, \ldots, t_i$ be the eigenspaces of $A_i$ for any $i$. 
Then  
 $$\C^n= V_1 \oplus \ldots \oplus V_k$$
 for some  linear subspaces $V_j$ such that:
 
 1) for any $i=1, \ldots, r$ and $l=1 , \ldots, t_i$ the subspace $E_l^i$ is the direct sum of some of the $V_j$,
 
 2) for any $j=1, \ldots, k$, we have that 
 $$V_j = E_{l_1}^{i_1} \cap \ldots \cap E_{l_u}^{i_u}$$ 
 for some $u, i_1, \ldots , i_u, l_1, \ldots, l_u$.
\end{lem}

\begin{proof} We prove the statement by induction on $r$.

If $r=1$ we can take $V_l= E_l^1$ for $l=1, \ldots, t_1$ and $k=t_1$.

Let us prove that the statement in the case $r-1$ implies the statement in the case $r$. 
By induction assumption, 
 $$\C^n= W_1 \oplus \ldots \oplus W_h$$
 for some  linear subspace $W_j$ such that:
 
 1) for any $i=1, \ldots, r-1$ and for any $l=1, \ldots, t_i$ the subspace $E_l^i$ is the direct sum of some of the $W_j$
 
 2) for any $j=1, \ldots, h$, we have that 
 $$W_j = E_{l_1}^{i_1} \cap \ldots \cap E_{l_u}^{i_u}$$ 
 for some $u, i_1, \ldots , i_u, l_1, \ldots, l_u$ with $i_1, \ldots, i_u \in \{1, \ldots, r-1\}$.
 
 By 2)  and by the fact that $A_r(E_l^i) \subset E_l^i$ for any $i=1, \ldots, r-1$ and any $l=1, \ldots, t_i$, we have that $A_r (W_j) \subset W_j$ for any $j=1, \ldots,h$. Hence the eigenspaces of $A_r$ are the direct sum of the eigenspaces of $A_r|_{W_j}$ for $j=1, \ldots, h$.
Consider the subspaces $$W_j \cap E_l^r,$$ 
for  $j=1, \ldots, h$ and $ l=1, \ldots, t_r$, that are nonzero; order them in some way and call them $V_1, \ldots, V_k$. 
Obviously every $W_j$ is the direct sum of some of the $V_m$, so 
 $$\C^n= V_1 \oplus \ldots \oplus V_k$$
and condition 1) is clearly satisfied.  Also condition 2) is satsified, in fact: let $m \in \{1, \ldots, k\}$; let $j$ and $l$ be such that 
$ V_m = W_j \cap E_l^r$;
 since $$W_j = E_{l_1}^{i_1} \cap \ldots \cap E_{l_u}^{i_u}$$
 for some $u, i_1, \ldots , i_u, l_1, \ldots, l_u$ with $i_1, \ldots, i_u \in \{1, \ldots, r-1\}$, we get $$ V_m = E_{l_1}^{i_1} \cap \ldots \cap E_{l_u}^{i_u} \cap E_l^r,$$
 hence condition 2) is satisfied. 
\end{proof}

\begin{cor}  \label{simdiag}
If $\{A_1, \ldots, A_r\}$ is a  subset of $M(n \times n, \C) $ such that $A_i$ and $A_j$ are 
simultaneously diagonalizable for every $i,j \in \{1, \ldots,r\}$, then  there  exists a basis ${\cal B}$ of  $\C^n$ 
such that every element of ${\cal B}$ is an  eigenvector of $A_i$  for every $i \in \{1, \ldots,r\}$.
\end{cor}

\begin{proof}
Write 
 $$\C^n= V_1 \oplus \ldots \oplus V_k$$
 as in Lemma \ref{simdiagpre}; for $j=1,\ldots, k$, take a basis of $V_j$ and let ${\cal B}$ be the union of such bases.
\end{proof}
\begin{cor} \label{simdiaginf}
If $Z$ is a  linear subspace of $M(n \times n, \C) $ such that $A,B$  are 
simultaneously diagonalizable for every $A,B \in Z$, then  there  exists a basis ${\cal B}$ of  $\C^n$ 
such that every element of ${\cal B}$ is an  eigenvector of every element of $Z$.
\end{cor}

\begin{proof}
Let $\{A_1, \ldots, A_r\}$ be a basis of $Z$;
by Corollary \ref{simdiag} there  exists a basis ${\cal B}$ of  $\C^n$ 
such that every element of ${\cal B}$ is an  eigenvector of $A_i$  for every $i \in \{1, \ldots,r\}$. But then every element of ${\cal B}$ is an  eigenvector of any linear combination of the matrices  $A_i$, that is 
of every element of $Z$.
\end{proof}

As we have already said the proof of Theorem  \ref{nilp} is very similar to  the proof of Gerstenhaber's result in \cite{M-O-R} but we need to use Lemma \ref{generaliz}. 

\begin{proof}[Proof of Theorem \ref{nilp}]
Let $P \in M(n \times n,K)$ and $Z$ linear subspace of $M(n \times n,K)$.
Let $S=P+Z$ and suppose $S \subset N(n,K)$. We want to show that $\dim(S)$, that is $\dim(Z)$, is less than or equal to   $\frac{n(n-1)}{2}.$ We can suppose that $P$ is in Jordan form.

Let $T$ be the set of the strictly upper triangular matrices.

Consider the bilinear form on  $M(n \times n,K)$
$ (A,B) \mapsto tr(AB).$ 
It is nondegenerate and, with respect to this bilinear form, we have that 
$$ T^{\perp}=\{A \in M(n \times n, K) | \;A\; \mbox{\rm upper triangular}\}.$$ 
Let $Z_1= Z \cap T$
and let $Z_2$ be a subspace such that 
$$ Z= Z_1 \oplus Z_2.$$
Observe that   
\begin{equation} \label{form1}
Z_2 \cap T^{\perp} \subset T, 
\end{equation}
in fact if $A$ is an upper triangular matrix such that $P+A$ is nilpotent, then $A$ is strictly upper triangular (remember that $P$ is in Jordan form and nilpotent).

Moreover, for the definition of $Z_1$ and $Z_2$,  we have that
$Z_2 \cap T=\{0\}.$
From this formula and from formula  (\ref{form1}), we get:
\begin{equation} \label{form3}
Z_2 \cap T^{\perp}=\{0\}.
\end{equation}
Obviously, since $Z_1 \subset T$, we have that 
\begin{equation} \label{form4}
T^{\perp} \subset Z_1^{\perp}.
\end{equation}
Finally, from Lemma \ref{generaliz}, we have that 
\begin{equation} \label{form5}
Z_2 \subset Z_1^{\perp}
\end{equation}
From   (\ref{form3}), (\ref{form4}) and (\ref{form5}), we get:
$$ Z_2 \oplus T^{\perp} \subset Z_1^{\perp},$$ 
hence 
$$ \dim(Z_2) + \dim( T^{\perp} ) \leq\dim( Z_1^{\perp})= n^2 -\dim (Z_1),$$
thus $$\dim (Z)  \leq  n^2 - \dim( T^{\perp} )= 
\frac{n(n-1)}{2} .$$

Suppose now that  the characteristic of $K$ is $0$ and that $P \not \in Z$. Let $\{z_1, \ldots, z_h\}$ be a basis of $Z$. Then the span of $P$ and $Z$ is generated by $P, P+z_1, \ldots, P+z_h$. The additive semigroup generated by 
$P, P+z_1, \ldots, P+z_h$ consists only of nilpotents.   
 Theorem 3 in the paper \cite{M-O-R} states that  the linear space generated by a set ${\cal E}$ of $ n \times n$ matrices over a field of characteristic $0$  consists only of nilpotents if and only if  the additive semigroup generated  ${\cal E}$  consists only of nilpotents. Hence the linear space generated by 
$P, P+z_1, \ldots, P+z_h$, that is the span of $P$ and $Z$, consists only of nilpotents.   Hence, by Theorem \ref{Gerst} the dimension of the span of $P$ and $Z$ is less than or equal to $
\frac{n(n-1)}{2} $, thus the dimension of $Z$ is  less than or equal to $
\frac{n(n-1)}{2}-1 $.

\end{proof}

\begin{rem} The second statement of Theorem \ref{nilp} is not true if the characteristic of the field  $K$ is not zero: take $K= \Z/2$, $n=2$ and $S = E_{1,2}+ \langle E_{1,2}+ E_{2,1}\rangle$; it consists only of $E_{1,2}$ and of $E_{2,1}$, which are both nilpotent, and its dimension is $1$, that is $
\frac{n(n-1)}{2} $.

\end{rem}

\begin{proof}[Proof of Theorem \ref{norm}]
Let $P \in M(n \times n,\C)$ and $Z$ linear subspace of $M(n \times n,\C)$.
Let $S=P+Z$ and suppose $S \subset R(n)$. We want to show that $\dim(S)$, that is $\dim(Z)$, is  less than or equal to  $n$. 

Let $A \in Z$; then 
$$( P+s A)(P^{\ast}+s A^{\ast})= (P^{\ast}+s A^{\ast})( P+s A) \;\;\;\; \; \forall s \in \R ,$$
and this implies 
$$P P^{\ast}+ s( AP^{\ast}+ PA^{\ast}) + s^{2} AA^{\ast} =  P^{\ast} P+ s( P^{\ast}A+ A^{\ast}P) + s^{2} A^{\ast} A \;\;\;\; \; \forall s \in \R ,$$
hence (since $P$ is normal)
$$s( AP^{\ast}+ PA^{\ast}) + s^{2} AA^{\ast} =  s( P^{\ast}A+ A^{\ast}P) + s^{2} A^{\ast} A \;\;\;\; \; \forall s \in \R ,$$
in particular, if we take $s=1$ and $s=-1$, we get that $A$ is normal. Hence $Z \subset R(n)$. 

So, to prove our statement, it is sufficient to prove that, if $Z$ is a linear subspace in $R(n)$, then $\dim(Z) \leq n$.

Let $A,B \in Z$. Hence $A+z B \in Z \subset R(n)$ for any $z \in \C$, thus 
$$ (A+zB)(A^{\ast} + \overline{z}B^{\ast}) = 
 (A^{\ast} + \overline{z}B^{\ast}) (A+ zB), $$
that is 
$$ A A^{\ast} -A^{\ast} A + z (B A^{\ast}-A^{\ast}B) + \overline{z}(AB^{\ast}-B^{\ast}A) + |z|^2 (BB^{\ast}-B^{\ast}B)=0.  $$
Since $A$ and $B$ are normal, we get that, for any $r, s\in \R$, 
$$  (r+is) (B A^{\ast}-A^{\ast}B) + (r-is)(AB^{\ast}-B^{\ast}A)=0 .$$
If we take first $r=1$ and $s=0$ and then 
$r=0$ and $s=1$, we get respectively
$$   (B A^{\ast}-A^{\ast}B) + (AB^{\ast}-B^{\ast}A )=0$$
and
$$  (B A^{\ast}-A^{\ast}B)  -(AB^{\ast}-B^{\ast}A) =0.$$
Therefore
$$  B A^{\ast}-A^{\ast}B  =AB^{\ast}-B^{\ast}A =0.$$ 
In particular $A$ and $B^{\ast}$ are simultaneously diagonalizable and so, being normal, they are unitarily simultaneously diagonalizable. Since a unitary matrix diagonalizing $B^{\ast}$ diagonalizes also $B$,  
we get that  $A$ and $B$ are simultaneously diagonalizable. 
So we have proved that every couple of matrices in $Z$ is  simultaneously diagonalizable. By Corollary \ref{simdiaginf}
there is a basis ${\cal B}$ of  $\C^n$ 
such that every element of ${\cal B}$ is an  eigenvector of every element of $Z$.
Let $C$ be an $(n \times n)$-matrix whose columns are the elements of ${\cal B}$. 
Hence $C^{-1} ZC$ is a  linear subspace of diagonal matrices. So $$\dim (Z)= \dim (C^{-1} ZC) \leq n.$$
\end{proof}

\begin{proof}[Proof of Theorem \ref{diag}]
The first statement is obvious, in fact a linear subspace contained in $D(n,\R) $ can intersect $A(n, \R)$
 only in $0$, so its dimension must be less than or equal to $\frac{n(n+1)}{2}$ and 
the linear subspace of the symmetric matrices achieves this dimension.

Now let $S$ be an affine not linear subspace contained in $D(n,\R)$;
 let $$S = P+Z,$$ where $Z$ is a linear subspace (and  obviously $ P \not \in Z$).
We want to show that $$\dim(Z) \leq  \frac{n(n+1)}{2} -1.$$ 
Let $W$ be the span of $P$ and $Z$. If the dimension of $Z$ were greater than or equal to $ \frac{n(n+1)}{2} $, then the dimension of $W$ would be greater than or equal to $ \frac{n(n+1)}{2} +1 $, so the intersection of $W$ and $A(n, \R)$  would contain a nonzero matrix $Y$.

Since $A(n,\R) \cap D(n,\R)=\{0\}$ and $W \setminus Z \subset D(n,\R) $, we would have that $Y \in Z$.  Hence $P+tY $ would be in $S$ for every $t \in \R$, therefore it would be  diagonalizable for every $t \in \R$. But $Y$ is nonzero and antisymmetric, hence it has at least a nonzero pure imaginary eigenvalue; so, since the roots of a complex  polynomial vary continously as a function of the coefficients (see \cite{H-M}),
 for a sufficient large $t$ at least one of  the complex eigenvalues of $P+tY $ is not real, hence $P+tY $ cannot be diagonalizable. So we get a contradiction, hence we must have  $\dim (S)= \dim(Z) \leq  \frac{n(n+1)}{2} -1.$
 
 Now let $U$ be the span of the matrices $E_{i,i}$ for $i=2, \ldots, n$ and of $E_{i,j}+ E_{j,i}$ for $i,j \in \{1, \ldots, n\}$ with $i <j$.
 Obviously $E_{1,1}+U$ is an affine not linear subspace of dimension $\frac{n(n+1)}{2} -1$
 whose elements are symmetric and hence diagonalizable. Thus
 we have proved that the 
maximal dimension of an affine not linear subspace in $D(n, \R)$ is  $ \frac{n(n+1)}{2} -1$.
\end{proof}

\begin{rem}
One can wonder if the only linear subspace of maximal dimension in $D(n, \R)$ is the subspace of the symmetric matrices, but the answer is obviously  no: for instance the linear subspace $\left\{ \bp a & b \\ 2b & c\ep| \; a,b,c \in \R \right\}$ 
is contained in $D(2, \R)$ and has dimension $3$.
\end{rem}

{\bf Acknowledgments.}
This work was supported by the National Group for Algebraic and Geometric Structures, and their  Applications (GNSAGA-INdAM).



{\small \begin{thebibliography}{Dilloo Dilloo 83}

\bibitem{Fl} H. Flanders, {\em On spaces of linear transformations with bounded rank},
J. London Math. Soc., 37 (1962),  pp. 10-16.

\bibitem{Ga} F.R. Gantmacher, {\em The Theory of Matrices}, vol. 2 Chelsea, Publishing Company New York 1959.


\bibitem{H-M} G. Harris, C. Martin {\em The roots of a complex  polynomial vary continously as a function of the coefficients}, Proc. A.M.S., 100 (1987), pp. 390-392.


\bibitem{H-P} W.V.D. Hodge, D. Pedoe, {\em Methods of Algebraic Geometry}, vol. 2 Cambridge University PressCambidge 1994. 



\bibitem{I-L} B. Ilic, J.M.  Landsberg {\em On symmetric degeneracy loci, spaces of symmetric matrices of constant rank and dual varieties}, Mathematischen Annlen, 
 314 (1999),  pp. 159-174. 

\bibitem{DP} C. De Seguins Pazzis {\em Large affine spaces of non-singular matrices}, Trans. A.M.S., 365 (2013), pp. 2569-2596.


\bibitem{DP2} C. De Seguins Pazzis
{\em On the matrices of given rank in a large subspace}, Linear Algebra Appl., 435 (2011), pp. 147-151.


\bibitem{G} M. Gerstenhaber {\em On Nilalgebras and Linear Varieties of Nilpotent Matrices, I}, Amer. J. Math., 80 (1958), pp. 614-622.

\bibitem{M-O-R} B. Mathes, M. Omladi\u{c}, H. Radjavi
{\em Linear subspaces of nilpotent matrices}, Linear Algebra Appl., 149 (1991), pp. 215-225.

\bibitem{Q} R. Quinlan
{\em Spaces of matrices without non-zero eigenvalues in their field of definition, and a question of Szechtman}, Linear Algebra Appl., 434 (2011), pp. 1580-1587.


\bibitem{Ru} E. Rubei {\em 
Affine subspaces of matrices with constant rank}, Linear  Algebra Appl., 644  (2022), pp.
 259-269. 


\bibitem{Ru2} E. Rubei {\em Affine subspaces of antisymmetric matrices with constant rank},  	arXiv:2209.07633,  to appear in Linear and Multilinear Algebra, DOI 10.1080/03081087.2023.2198759
 

\bibitem{W1} 
R. Westwick {\em 
Spaces of matrices of fixed rank}, Linear and Multilinear Algebra, 20  (1987), pp. 
171-174. 




\end{thebibliography}}
\end{document}